\documentclass[12pt]{elsarticle}

\newtheorem{thm}{Theorem} 
 
 \newtheorem{lem}[thm]{Lemma}

 \newdefinition{rem}{Remark}
 
\newproof{proof}{Proof}
\newtheorem{defn}{Definition}

\usepackage{amsmath}
\usepackage{amsfonts}
\usepackage{amssymb}

\newcommand{\abs}[1]{\ensuremath{\left| #1 \right| }}
\newcommand{\qPs}{$q$-Pochhammer symbol}

\newcommand{\RRis}{Rogers--Ramanujan identities}

\date{}
\journal{arXiv}

\begin{document}

\begin{frontmatter}

\title
{Combinatorial Formulas for Certain Sequences of Multiple Numbers}

\author{Hasan Coskun}
\ead{hasan.coskun@tamuc.edu}
\address{Department of Mathematics, Texas A\&M
  University--Commerce, Binnion Hall, Room 314, Commerce, TX 75429}  

\begin{abstract}
Multiple analogues of certain families of combinatorial numbers are recently constructed by the author in terms of well poised Macdonald functions, and some of their fundamental properties are developed. In this paper, we present combinatorial formulas for the well poised Macdonald functions, the multiple binomial coefficients, the multiple bracket function, and the multiple Catalan and Lah numbers. 
\end{abstract}


\end{frontmatter}

\section{Introduction}
\label{introduction}
The author has recently constructed multiple $qt$-analogues of several sequences of 
classical combinatorial numbers including the factorial function, the binomial coefficients, Stirling, Lah, Fibonacci, Bernoulli, Catalan, and Bell number sequences~\cite{CoskunDiscM}. We have also developed certain analytic and algebraic properties of these sequences including their generating functions and recurrence relations~\cite{Coskun3}. In this paper, we focus on another important aspect of these sequences, namely the combinatorial formulas that they satisfy. More specifically, we write combinatorial formulas in terms of reversed Young tableau for the well poised Macdonald functions $w_\lambda$, and the multiple analogues of the factorial function, the binomial coefficients, the Lah numbers and the Catalan numbers. 

Our multiple extensions of these sequences are written in terms of limiting cases of the well poised Macdonald functions~\cite{Coskun1}, which are equivalent to the shifted Macdonald polynomials $P^{*}_\lambda$ developed by Sahi and Knop~\cite{Sahi0, KnopSahi}. Besides studying other properties, Okounkov had developed a combinatorial formula for these polynomials~\cite{Okounkov1} which may be written as follows:
\begin{equation}
\label{combPstar}
P^*_{\lambda}(x; q, t) \\
= \sum_{T} \psi_T(q,t) \prod_{s\in\lambda} t^{1-T(s)} (x_{T(s)} - q^{a'(s)} t^{-l'(s)})
\end{equation}
where the sum is over all reversed Young tableau on $\lambda$ (that are semistandard, decreases strictly down the columns, and weakly in rows) with entries $T(s)$ for each square $s=(i,j)$ of $\lambda$ from the set $\mathbb{N}:=\{1,2, \ldots \}$. 
Here, the $\psi_T(q,t)$ is the same weight factor~(\ref{psiH}) for each tableaux $T$ used in the ordinary Macdonald polynomials~\cite{Macdonald1}, and the arm colength $a'(s)=j-1$ and leg colength $l'(s)=i-1$ are defined as usual. 

The main ingredient in his construction is the recurrence formula (branching rule) for the $P^{*}_\lambda$ functions. 
Using the recurrence relation for our $w_\lambda$ functions in a similar manner, we first write down a combinatorial formula for the well poised Macdonald functions. This allows us to construct combinatorial formulas for several sequences of numbers listed above in terms of $w_\lambda$ functions. For example, with the same notation as above, we write the formula for the bracket function as
\begin{multline*}
[z, s]_{\lambda} = 
\prod_{s\in \lambda} \dfrac{ 1 }{(1- q t^{n-1-l'(s)})  }   
\prod_{s\in \lambda} \left( 
\dfrac{1-q^{a_\lambda(s)} t^{l_\lambda(s)+1}}
{1 - q^{a'_\lambda(s)} t^{- l'_\lambda(s)+n}}  \right)  \\ \cdot
\sum_{T} \psi_T(q,t) \prod_{s\in\lambda}  t^{ -l'(s) +n -T(s)} 
(1 - s_{T(s)}  q^{z_T(s) -a'(s) } t^{l'(s) }   ) 
\end{multline*}
where 
$z, s\in\mathbb{C}^n$ are $n$-tuples of complex numbers, 
and the index partition $\lambda$ is of length at most $n$ (i.e., $\ell(\lambda)\leq n$). 
Note $[z, s, n, q, t]_\lambda = 0$ if $\ell(\lambda)>n$, and $[z, s, n, q, t]_\lambda \neq 0$ otherwise. 

\section{Background}
\label{back}

The multiple generalizations of combinatorial number sequences developed in~\cite{CoskunDiscM} are written in terms of the rational Macdonald functions. We therefore start with a quick review of these functions. 

The (basic) \qPs\ $(a;q)_\alpha$ may be defined formally for complex parameters $q, \alpha\in\mathbb{C}$ as
\begin{equation}
\label{qPochSymbol}
(a)_\alpha = (a;q)_\alpha :=\dfrac{(a;q)_\infty}{(aq^\alpha;q)_\infty}
\end{equation} 
where the infinite product $(a;q)_\infty$ is defined by $(a;q)_\infty:=\prod_{i=0}^{\infty} (1-aq^i)$. Note that when $\alpha=m$ is a positive integer, the definition reduces to the finite product $(a;q)_m= \prod_{k=0}^{m-1}(1-aq^k)$. 
For any partition $\lambda = (\lambda_1, \ldots, \lambda_n)$ and
$t\in\mathbb{C}$, define~\cite{Warnaar2}
\begin{equation}
\label{ellipticQtPocSymbol}
  (a)_\lambda=(a; q, t)_\lambda := \prod_{i=1}^{n} (at^{1-i};
  q)_{\lambda_i} .
\end{equation}
Note that when $\lambda=(\lambda_1) = \lambda_1$ is a single part
partition, then $(a; q, t)_\lambda = (a; q)_{\lambda_1} =
(a)_{\lambda_1}$. 

\subsection{Well--poised Macdonald functions}

The construction of our multiple number sequences involve the 
(basic) well--poised Macdonald functions $W_{\lambda/\mu}$
on $BC_n$~\cite{Coskun1}. 
This remarkable family of symmetric rational functions appeared  
in~\cite{CoskunG1} in the most general elliptic form.   

Let $\lambda=(\lambda_1, \ldots, 
\lambda_n)$ and $\mu=(\mu_1, \ldots, \mu_n)$ be partitions of at most
$n$ parts for a positive integer $n$ such that the
skew partition $\lambda/\mu$ is a horizontal strip; i.e. $\lambda_1
\geq \mu_1 \geq\lambda_2 \geq \mu_2 \geq \ldots \lambda_n \geq
\mu_n \geq \lambda_{n+1} = \mu_{n+1} = 0$. Following~\cite{Coskun1}, we define
\begin{multline}
\label{definitionHfactor}
H_{\lambda/\mu}(q,t,b) 
:= \prod_{1\leq i < j\leq n} 
\left\{\dfrac{(q^{\mu_i-\mu_{j-1}}t^{j-i})_{\mu_{j-1}-\lambda_j}
(q^{\lambda_i+\lambda_j}t^{3-j-i}b)_{\mu_{j-1}-\lambda_j}}
{(q^{\mu_i-\mu_{j-1}+1}t^{j-i-1})_{\mu_{j-1}-\lambda_j}(q^{\lambda_i
    +\lambda_j+1}t^{2-j-i}b)_{\mu_{j-1}-\lambda_j}}\right.\\
\left.\cdot 
\dfrac{(q^{\lambda_i-\mu_{j-1}+1}t^{j-i-1})_{\mu_{j-1}-\lambda_j}}
{(q^{\lambda_i-\mu_{j-1}}t^{j-i})_{\mu_{j-1}-\lambda_j}}\right\} 
\cdot\prod_{1\leq i <(j-1)\leq n} \!\!\!
\dfrac{(q^{\mu_i+\lambda_j+1}t^{1-j-i}b)_{\mu_{j-1}-\lambda_j}}
{(q^{\mu_i+\lambda_j}t^{2-j-i}b)_{\mu_{j-1}-\lambda_j}}
\end{multline}
and 
\begin{multline}
\label{definitionSkewW}
W_{\lambda/\mu}(x; q,t,a,b)
:= H_{\lambda/\mu}(q,p,t,b)\cdot\dfrac{(x^{-1}, ax)_\lambda
  (qbx/t, qb/(axt))_\mu}
{(x^{-1}, ax)_\mu (qbx, qb/(ax))_\lambda}\\
\cdot\prod_{i=1}^n\left\{\dfrac{(1-bt^{1-2i}q^{2\mu_i})}{(1-bt^{1-2i})}
  \dfrac{(bt^{1-2i})_{\mu_i+\lambda_{i+1}}}
{(bqt^{-2i})_{\mu_i+\lambda_{i+1}}}\cdot
t^{i(\mu_i-\lambda_{i+1})}\right\}
\end{multline}
where $q,t,x,a,b\in\mathbb{C}$.  
Note that $W_{\lambda/\mu}(x; q, t, a,b)$ vanishes unless $\lambda/\mu$ is a horizontal strip. 
The function
$W_{\lambda/\mu}(y, x_1, \ldots, 
x_\ell; q,t,a,b)$ is extended to $\ell+1$ variables $y, x_1, \ldots, x_\ell
\in\mathbb{C}$  
through the following recursion formula
\begin{multline}
\label{eqWrecurrence}
W_{\lambda/\mu}(y,x_1,x_2,\ldots,x_\ell;q, t, a, b) \\
= \sum_{\mu\prec \nu\prec \lambda} W_{\lambda/\nu}(yt^{-\ell};q, t, at^{2\ell},
bt^\ell) \, W_{\nu/\mu}(x_1,\ldots, x_\ell;q, t, a, b).
\end{multline}

\subsection{The Limiting $w_{\lambda/\mu}$ Function}

The Macdonald functions $W_\lambda$ are essentially equivalent to $BC_n$ abelian functions constructed independently in~\cite{Rains1}. 
Likewise, the limiting cases defined below are equivalent to shifted Macdonald polynomials~\cite{Sahi0, KnopSahi, Okounkov1} which are themselves extensions of the Macdonald polynomials~\cite{Macdonald1}.

The following limiting 
 the basic (the $p=0$ case of the elliptic) $W_{\lambda/\mu}$ functions
will be used in our constructions below. 
The existence of these limits can be seen from 
the definition~(\ref{definitionSkewW}),  the recursion
formula~(\ref{eqWrecurrence}) and the limit rule
\begin{equation}
\label{LimitRule}
  \lim_{a\rightarrow 0}\, a^{|\mu|} (x/a)_{\mu} 
= (-1)^{|\mu|}\, x^{|\mu|} t^{-n(\mu)} q^{n(\mu')} 
\end{equation}
where $\abs{\mu}=\sum_{i=1}^n \mu_i$ and $n(\mu) =
\sum_{i=1}^n (i-1) \mu_i$, 
and $n(\mu') =\sum_{i=1}^n \binom{\mu_i}{2}$. 
We denote $H_{\lambda/\mu}(q,t) :=H_{\lambda/\mu}(q,t,0)$, and for $x\in \mathbb{C}$ define
\begin{multline}
\label{wdefn}
w_{\lambda/\mu}(x; q,t) := \lim_{s\rightarrow \infty} \left( s^{|\lambda|-|\mu|} \lim_{a\rightarrow 0} W_{\lambda/\mu}(x; q,t, a, as) \right) \\
= (-q/x)^{-|\lambda|+|\mu|} q^{ -n(\lambda') + n(\mu') }  H_{\lambda/\mu}(q, t) \dfrac{(x^{-1})_\lambda }{(x^{-1})_\mu }  
\end{multline}
The recurrence formula for $w_{\lambda/\mu}$ 
function turns out to be
\begin{equation}
\label{wsuprec}
w_{\lambda/\mu}(y,z;q, t) \\
= \sum_{\nu\prec \lambda} t^{\ell(|\lambda|-|\nu|)}
w_{\lambda/\nu}(yt^{-\ell};q, t) \, 
w_{\nu/\mu}(z;q, t)
\end{equation}

\begin{rem}
\label{wdual}
We will need the following properties from~\cite{CoskunDiscM} in what follows.
Let $\mu$ be a partition of at most $n$ part, and $z=(x_1,\ldots, x_n)\in\mathbb{C}^n$. Then  \\

\noindent
(3)
Let $z=\bar{x} = (x, x, \ldots, x) \in\mathbb{C}^n$ for $x\in\mathbb{C}$, then we have
\begin{multline}
w_{\mu}(\bar{x} t^{\delta(n)};q,t)  
=  (-1)^{|\mu|} x^{|\mu|} t^{n(\mu)} q^{-|\mu|-n(\mu')} (x^{-1})_\mu
\! \prod_{1\leq i < j\leq n} \dfrac{(t^{j-i+1})_{\mu_i
-\mu_j} } {(t^{j-i})_{\mu_i -\mu_j} }  
\end{multline}
which, after flipping $q$ and $t$ and using the flip rule, 
\begin{equation}
\label{flip}
x^{|\mu|}   (x^{-1}, q, t)_\mu 
= (-1)^{|\mu|} q^{n(\mu')} t^{-n(\mu)} (x;q^{-1}, t^{-1})_{\mu}
\end{equation}
may be written as
\begin{equation}
\label{wsxtdelta}
w_{\mu}(\bar{x} t^{\delta(n)};q,t) 
=  q^{-|\mu|}  (x;q^{-1}, t^{-1})_{\mu}
\! \prod_{1\leq i < j\leq n} \dfrac{(t^{j-i+1})_{\mu_i
-\mu_j} } {(t^{j-i})_{\mu_i -\mu_j} } 
\end{equation}

\noindent
(4)
The vanishing property of $W_{\lambda}$ functions implies that 
\begin{equation}
\label{wvanish}
  w_{\mu}(q^\lambda t^\delta;q, t)=0
\end{equation}
when $\mu\not\subseteq \lambda$, where $\subseteq$ denotes the partial inclusion ordering.\\

\noindent
(5)
Let $\lambda$ be an $n$-part partition with $\lambda_n 
\neq 0$ and $0\leq k\leq \lambda_n$ for some integer $k$, and let 
$z=(x_1,\ldots,x_n)\in \mathbb{C}^n$. It was shown in~\cite{CoskunDiscM} that
\begin{equation}
\label{wrect}
w_{\bar k}(z;q,t) 
= q^{-nk} \prod_{i=1}^n (q^{1-k} x_i)_k 
\end{equation}
where $\bar k = (k,\ldots, k)\in \mathbb{C}^n$. \\

\noindent
(6)
With the same notation as above, we also have
\begin{equation}
\label{wnormal}
w_\mu(q^\mu t^{\delta(n)}; q, t) 
= q^{-|\mu|}\, t^{(n-1)|\mu|-2n(\mu)} \, ( qt^{n-1} )_{\mu} 
 \!\!\prod_{1\leq i < j\leq n} \!\!
\frac{(qt^{j-i-1})_{\mu_i-\mu_j} } {(qt^{j-i})_{\mu_i-\mu_j} }
\end{equation}

\end{rem}

\subsection{The Multiple $qt$-Binomial Coefficients}
\label{section1}
Recall that the multiple $qt$-binomial coefficient is defined in~\cite{CoskunDiscM} as follows. 
\begin{defn}
\label{qtbinomcoeffExt}
Let $z=(x_1,\ldots, x_n)\in\mathbb{C}^n$, and $\mu$ be a partition of at most $n$-parts. Then the multiple $qt$-binomial coefficient is defined by  
\begin{equation}
\label{qtbinom}
\binom{z}{\mu}_{\!\!\!q,t} := \dfrac{ q^{|\mu|} t^{2n(\mu)+(1-n)|\mu| } }  { (qt^{n-1} )_\mu} \prod_{1\leq i<j \leq n} \left\{\dfrac{ (qt^{j-i})_{\mu_i-\mu_j} } {(qt^{j-i-1})_{\mu_i-\mu_j}  } \right\}  w_\mu(q^z t^{\delta(n)}; q, t)
\end{equation}
where $q,t\in\mathbb{C}$. It should be noted that this definition makes sense even for $\mu\in\mathbb{C}^n$. 
\end{defn}

Many interesting properties of the multiple binomial coefficients are obtained in~\cite{CoskunDiscM}. We will need below a special case when $\mu$ is a rectangular partition, that is $\mu=\bar k$. Using~(\ref{wrect}) we get
\begin{equation}
\label{binom_rect1}
\binom{z}{\bar k}_{\!\!\!q,t} =\prod_{i=1}^n \dfrac{ (q^{1-k} q^{x_i}t^{n-i} )_k  }  { (qt^{n-i} )_{k}}  
\end{equation}
In the particular case for $k=1$, the definition reduces to
\begin{equation}
\label{binom_rect2}
\binom{z}{\bar 1}_{\!\!\!q,t} =\prod_{i=1}^n \dfrac{ ( q^{x_i}t^{n-i})_1  }  { (qt^{n-i} )_{1}}  
= \prod_{i=1}^n \dfrac{(1-q^{x_i} t^{n-i} )}{(1-qt^{n-i} )} 
\end{equation}

\subsection{The Multiple $qt$-Factorial Function}
We now recall another important extension from~\cite{CoskunDiscM} that generalizes the one dimensional $q$-brackets and $q$-factorial polynomials to the multiple case. 
\begin{defn}
Let $\mu$ be a partition of at most $n$ parts, $z=(x_1,\ldots,x_n)\in\mathbb{C}^n$ and $s\in\mathbb{C}^n$. Then 
\begin{multline}
\label{factorialfunc}
[z, s]_{\mu} = [z, s, n, q, t]_\mu \\
:= q^{|\mu|} 
\prod_{i=1}^n  \left\{ \dfrac{1}{(1-qt^{n-i} )^{\mu_i}} \right\}
\prod_{1\leq i<j \leq n} \left\{ \dfrac{ (t^{j-i})_{\mu_i-\mu_j} } {(t^{j-i+1})_{\mu_i-\mu_j}  } \right\}  \, w_\mu(s q^z t^{\delta(n)}; q, t)  
\end{multline}
is called the $qt$-factorial (bracket) function. 
Note that the definition involves a multiplicative variable $s$, and an exponential variable $z$. Depending on the application we often set $z=\bar 0$ and write $\langle s\rangle_{\mu} = [\bar 0, s]_{\mu}$, or set $s=\bar 1$ and write $[ z]_{\mu} = [z, \bar 1]_{\mu}$. Using the identity~(\ref{wrect}) in the special case when $\mu= \bar 1$, we define the $qt$-bracket as
\begin{equation}
\label{qtnumber}
[z] =  [z, \bar 1,  n, q, t]_{\bar 1} 
= \prod_{i=1}^n \dfrac{(1-q^{x_i} t^{n-i} )}{(1-qt^{n-i} )} 
\end{equation}
which is a multiple analogue of the classical $q$-number or $q$-bracket.
\end{defn}

\begin{rem} 
The $qt$-factorial function satisfies the following properties: \\

\noindent
(a) Note $[z, s, n, q, t]_\lambda = 0$ if $\ell(\lambda)>n$, and $[z, s, n, q, t]_\lambda \neq 0$ otherwise. \\

\noindent
(b) Let $z=(x,\ldots, x)=\bar x  \in \mathbb{C}^n$ for a single variable $x\in\mathbb{C}$, then the $qt$-factorial function $[\bar x ]_{\mu}$ may be written as 
\begin{equation}
\label{spec_bracket}
[\bar x]_{\mu} 
=\prod_{i=1}^n  \left\{ \dfrac{1 }{(1-qt^{n-i} )^{\mu_i}} \right\} 
   (q^x;q^{-1}, t^{-1})_{\mu}
= \prod_{i=1}^n  \left\{ \dfrac{ (q^x t^{i-1};q^{-1})_{\mu_i} }{(1-qt^{n-i} )^{\mu_i}} \right\} 
\end{equation}
This definition reduces to the classical $q$-bracket in the one variable case. \\

\noindent
(c) Note that $(x;1/q,1/t)_\mu $, with the reciprocals of $q$ and $t$, corresponds to a multiple basic $qt$-analogue of the falling factorial $x_{\underline{n}}:= x (x-1) \cdots (x-(n-1))$ as opposed to the rising factorial or the Pochhammer symbol.  \\

\noindent
(d)
Setting $z=\mu$, and substituting the evaluation~(\ref{wnormal}) in~(\ref{qtnumberShifted}) above gives
\begin{multline}
[\mu]_{\mu} 
= t^{-2n(\mu)-(1-n)|\mu| } \,
\prod_{i=1}^n  \left\{ \dfrac{ (qt^{n-i} )_{\mu_i} }{(1-qt^{n-i} )^{\mu_i}} \right\} \\
\cdot   \prod_{1\leq i<j \leq n} \left\{\dfrac{ (t^{j-i})_{\mu_i-\mu_j} } {(t^{j-i+1})_{\mu_i-\mu_j}  }     \dfrac{(qt^{j-i-1})_{\mu_i-\mu_j}  } { (qt^{j-i})_{\mu_i-\mu_j} } \right\} \hspace{20pt}
\end{multline}
Similar to the classical case, we may use the notation $\mu!= [\mu]_{\mu} $ and write
\begin{equation}
\label{qtnumberShifted3}
[z]_{\mu} 
= \mu! \, \binom{z}{\mu}_{\!\!\!q,t}  
\end{equation}
Note that in the particular case when $\mu=\bar k$ is a rectangular partition, the $\mu!$ reduces to a product of one dimensional quotients for each part.
\begin{equation}
\label{rect!}
\bar k !
= \prod_{i=1}^n  \left\{ \dfrac{ (qt^{n-i} )_{k} }{(1-qt^{n-i} )^{k}} \right\} 
\end{equation}

\end{rem}

\section{A combinatorial formula for the $w_\lambda$ functions}
\label{combMacdonald}

Recall that $\mu\preccurlyeq \lambda$ means that, for each $i\in [n]$
\[ \lambda_i \geq \mu_i, \quad\mathrm{and} \quad  \mu_i \geq \lambda_{i+1}. \]
For a tableux $T$ of at most $n$ rows (whose shape is a partition $\lambda$ with at most $n$ parts), we have
\[ \psi_T = \prod_{i=1}^n \psi_{\lambda^{(i)}/\lambda^{(i-1)} } \]
where 
$\emptyset = \lambda^{(0)} \preccurlyeq \cdots \preccurlyeq \lambda^{(n)}=\lambda$ is the decomposition sequence for $T$, and 
$\psi_{\lambda/\mu}$ is defined combinatorially as
\begin{equation*}
\psi_{\lambda/\mu}  = \prod_{s\in R_{\lambda/\mu}  - C_{\lambda/\mu} } \dfrac{b_\mu(s)}{b_\lambda(s) }
\end{equation*}
where $R_{\lambda/\mu}$ and $C_{\lambda/\mu}$ denotes the rows and columns that intersect the horizantal strip $\lambda/\mu$ in the Young diagram, and for each square $s=(i,j)\in\lambda$
\begin{equation}
\label{bfac}
b_\lambda(s) = \dfrac{1- q^{a_\lambda(s)} t^{l_\lambda(s)+1} }
{1-q^{a_\lambda(s)+1} t^{l_\lambda(s)}}
\end{equation}
Alternatively, an algebraic definition is given by   
\begin{equation}
\label{psiH}
\psi_{\lambda/\mu} = \prod_{1\leq i\leq j\leq \ell(\mu)}  
\dfrac{f(q^{\mu_i - \mu_j } t^{j-i} ) f(q^{\lambda_i - \lambda_{j+1} } t^{j-i} )}{f(q^{\lambda_i - \mu_j } t^{j-i} ) f(q^{\mu_i - \lambda_{j+1} } t^{j-i} )}
\end{equation}
where $f(a) =(at)_\infty / (aq)_\infty $. 

As pointed out in the Introduction, Okounkov gives a combinatorial formula for the interpolation Macdonald polynomials $P^*_\lambda$ in terms of the branching rule
\begin{multline}
P^*_{\lambda}(x_1,\ldots, x_n; q, t) \\
= \sum_{\mu\prec \lambda} \psi_{\lambda/\mu}(q,t)  \, t^{-|\mu|} 
\prod_{s\in \lambda/\mu} (x_1 - q^{a'(s)} t^{-l'(s)}) \, P^*_{\mu}(x_2,\ldots, x_n; q, t)
\end{multline}
where the sum is over all partitions that gives horizontal strips. Iteration of this result proves the combinatorial formula. It is clear that iterating this formula we obtain the semistandard tableaux sum formula~(\ref{combPstar}) for  $P^*_\lambda$.

We take a similar approach, and write a combinatorial formula for the $w_\lambda$ functions. First we write  combinatorial representations for certain products that occur often. 
\begin{lem}
\label{factors1}
The algebraic products on the left hand sides can be represented by combinatorial products on the right hand sides as follows: 
\begin{equation}
(a)\;  (x; q, t)_\lambda
= \prod_{s\in \lambda} (1- x q^{a'(s)}t^{-l'(s)} )   
\end{equation}
and 
\begin{equation}
(b)\; \prod_{i=1}^n (x_i t^{1-i})_{\lambda_{i}} = \prod_{s\in \lambda} (1- x_{1+l'(s)} q^{a'(s)} t^{-l'(s)} )
\end{equation}
and
\begin{equation}
(c)\;  x^{n(\lambda')} = \prod_{s\in \lambda} x^{a'(s)}
\qquad
(d)\;  x^{n(\lambda)} = \prod_{s\in \lambda} x^{l'(s)}
\qquad
(e)\;  x^{|\lambda|} = \prod_{s\in \lambda} x 
\end{equation}
and 
\begin{equation}
(f)\; \prod_{i=1}^n (1- y x^{n-i})^{\lambda_i}  = 
\prod_{s\in \lambda} (1- y x^{n-1-l'(s)})  
\end{equation}
We also have 
\begin{equation}
(g)\;  
x^{|\lambda^{(n)}|+\cdots+|\lambda^{(1)}|} = \prod_{s\in \lambda} x^{T(s)} 
\end{equation}
where $\emptyset = \lambda^{(0)} \prec \cdots \prec \lambda^{(n)}=\lambda$ is the decomposition sequence for the tableaux $T$, and $T(s)$ represents the filling in the square $s$ in $T$ as usual. 
\end{lem}

\begin{proof}
Proof follows from direct calculations. For example, 
\begin{equation}
\prod_{s\in \lambda} x^{a'(s)} = 
\prod_{i=1}^n x^{0} \cdots x^{(\lambda_i - 1)} 
= \prod_{i=1}^n x^{\binom{\lambda_i}{2}}  
= x^{\sum_{i=1}^n \binom{\lambda_i}{2}}
= x^{n(\lambda')}
\end{equation}
and
\begin{equation}
\prod_{s\in \lambda} x^{l'(s)} = 
\prod_{i=1}^n \left(x^{i-1} \right)^{\lambda_i} 
= x^{\sum_{i=1}^n (i-1)\lambda_i}
= x^{n(\lambda)}
\end{equation}
and, in particular
\begin{equation}
\prod_{s\in \lambda} x = 
\prod_{i=1}^n x^{\lambda_i} 
= x^{\sum_{i=1}^n \lambda_i}
= x^{|\lambda|}
\end{equation}
Note that (a) is simply a special case of (b). Other properties can be verified similarly. 
\end{proof}

Now we are ready to write the combinatorial formula for the symmetric $w_{\lambda/\mu}$ functions. 
\begin{thm}
\label{combformw}
Let $\lambda$ and $\mu$ be partitions of at most $n$ parts. The function $w_{\lambda/\mu}(x; q,t) $ of a single variable $x\in\mathbb{C}$ may be written combinatorially as
\begin{equation}
w_{\lambda/\mu}(x; q,t) 
=  \psi_{\lambda/\mu}(q, t)
\prod_{s\in \lambda/\mu} q^{-1} t^{-l'(s)}  (1- x q^{-a'(s)}t^{l'(s)} ) 
\end{equation}
In the multivariable case for $z=(x_1, x_2, \ldots, x_n)\in\mathbb{C}^n$, we have
\begin{equation}
w_\lambda(z; q, t)  
= \sum_{T} \psi_T(q,t) \prod_{s\in\lambda} ( - x_{T(s)} q^{-1-a'(s)} + q^{-1} t^{n -l'(s) -T(s)}  ) 
\end{equation}
where the sum is over all semistandard reversed Young tableau of shape $\lambda$.
\end{thm}

\begin{proof}
First use the identity 
\begin{equation}
\label{flip}
x^{|\mu|}   (x^{-1}, q, t)_\mu 
= (-1)^{|\mu|} q^{n(\mu')} t^{-n(\mu)} (x;q^{-1}, t^{-1})_{\mu}
\end{equation}
to write the definition~(\ref{wdefn}) of $w_{\lambda/\mu}$ in the form 
\begin{equation}
w_{\lambda/\mu}(x; q,t) 
= q^{-|\lambda|+|\mu|} t^{-n(\lambda)+n(\mu)}  H_{\lambda/\mu}(q, t) \dfrac
{ (x;q^{-1}, t^{-1})_{\lambda} }
{   (x;q^{-1}, t^{-1})_{\mu} }   
\end{equation}
for $x\in\mathbb{C}$. 
It is easy to see by some algebraic manipulations that 
the $H_{\lambda/\mu}(q, t)$ factor in the definition of $w_{\lambda/\mu}$ function
\begin{multline*}
\label{psiHexp}
H_{\lambda/\mu}(q,t) 
= \prod_{1\leq i < j\leq n} 
\left\{\dfrac{(q^{\mu_i-\mu_{j-1}}t^{j-i})_{\mu_{j-1}-\lambda_j} }
{(q^{\mu_i-\mu_{j-1}+1}t^{j-i-1})_{\mu_{j-1}-\lambda_j} }
\dfrac{(q^{\lambda_i-\mu_{j-1}+1}t^{j-i-1})_{\mu_{j-1}-\lambda_j}}
{(q^{\lambda_i-\mu_{j-1}}t^{j-i})_{\mu_{j-1}-\lambda_j}}\right\} 
\end{multline*}
is the same as the $\psi_{\lambda/\mu}$ factor ~(\ref{psiH}) above. That is, $H_{\lambda/\mu}(q, t)=\psi_{\lambda/\mu}(q, t)$. We now use Lemma~\ref{factors1} to write 
\begin{equation}
\label{onevar}
w_{\lambda/\mu}(x; q,t) 
=  \psi_{\lambda/\mu}(q, t)
\prod_{s\in \lambda/\mu} q^{-1} t^{-l'(s)}  (1- x q^{-a'(s)}t^{l'(s)} ) 
\end{equation}
which gives the first part of the Theorem. 

The recurrence formula for $w_{\lambda/\mu}$ function~(\ref{wsuprec}) maybe written as 
\begin{multline}
w_{\lambda/\mu}(x_1,x_2,\ldots, x_n;q, t) \\
= \sum_{\mu\prec \nu\prec \lambda} t^{(n-1) (|\lambda|-|\nu|)}
w_{\lambda/\nu}(x_1 t^{1-n};q, t) \, 
w_{\nu/\mu}(x_2,\ldots, x_n;q, t)
\end{multline}
Set $\mu=0$ and substitute the formula~(\ref{onevar}) into this recurrence to write 
\begin{multline}
w_{\lambda}(x_1,x_2,\ldots, x_{n};q, t) 
= \sum_{\nu\prec \lambda} \psi_{\lambda/\nu}(q, t)  
\prod_{s\in \lambda/\nu} (-q^{-1} )  (x_1 q^{-a'(s)} - t^{n-1-l'(s)} ) \\ \cdot 
w_{\nu}(x_2,\ldots, x_{\ell+1};q, t)
\end{multline}
Applying this argument to the $w_{\nu}$ function inside the sum on the right hand side repeatedly until all variables are decomposed, and simplifying gives the second part of the Theorem. 
\end{proof}

\section{Combinatorial formulas for multiple combinatorial numbers}
\label{combMacdonald}

We need combinatorial representations of certain double products that turn out to be special evaluations of the ordinary Macdonald polynomials in the following. The next lemma carries out these calculations. 

\begin{lem}
\label{factors2}
The double product factors on the left may be written combinatorially as 
\begin{equation}
(a) \; \prod_{1\leq i<j\leq n} 
\dfrac{(qt^{j-i})_{\lambda_i-\lambda_j}}
{(qt^{j-i-1})_{\lambda_i-\lambda_j} } 
= \prod_{s\in \lambda}  
\left( \dfrac{1 - q^{a'_\lambda(s)+1} t^{-l'_\lambda(s)+n-1}  }
{1-q^{a_\lambda(s)+1} t^{l_\lambda(s)}} \right)
\end{equation}
and
\begin{equation}
(b) \hspace{30pt} \prod_{1\leq i<j \leq n} 
\dfrac{ (t^{j-i})_{\lambda_i-\lambda_j} } 
{(t^{j-i+1})_{\lambda_i-\lambda_j}  } 
= \prod_{s\in \lambda} \left( 
\dfrac{1-q^{a_\lambda(s)} t^{l_\lambda(s)+1}}
{1 - q^{a'_\lambda(s)} t^{- l'_\lambda(s)+n}}  \right)
\end{equation}
\end{lem}

\begin{proof}
We use the algebra homomorphism 
$\varepsilon_{u,t}$ of Macdonald, and 
the identity (6.17) in his book~\cite{Macdonald1} for the ordinary Macdonald polynomials $P_\lambda(q,t)$ and its dual $Q_\lambda(q,t)$. Applying $\varepsilon_{qt^{n-1},t}$ to both sides of the identity 
\begin{equation}
Q_\lambda(q,t) = b_\lambda(q,t) P_\lambda(q,t)
\end{equation}
we get
\begin{multline}
\prod_{1\leq i<j\leq n} 
\dfrac{(qt^{j-i})_{\lambda_i-\lambda_j}}
{(qt^{j-i-1})_{\lambda_i-\lambda_j} } 
= b_\lambda(q,t)  
\prod_{s\in \lambda} 
\dfrac{1 - q^{a'_\lambda(s)+1} t^{-l'_\lambda(s)+n-1}  }
{1-q^{a_\lambda(s)} t^{l_\lambda(s)+1}} \\
= \prod_{s\in \lambda}  
\dfrac{1- q^{a_\lambda(s)} t^{l_\lambda(s)+1} }
{1-q^{a_\lambda(s)+1} t^{l_\lambda(s)}} 
\prod_{s\in \lambda} 
\dfrac{1 - q^{a'_\lambda(s)+1} t^{-l'_\lambda(s)+n-1}  }
{1-q^{a_\lambda(s)} t^{l_\lambda(s)+1}}
\end{multline}
since 
\begin{equation*}
b_\lambda(q, t) = \prod_{s\in \lambda} b_\lambda(s,q, t) 
\end{equation*}
where $b_\lambda(s,q, t)$ is given in~(\ref{bfac}). 
We cancel the like factors to complete the first part of the Theorem. 

The second part of the Theorem is similar. 
It suffices to apply $\varepsilon_{t^{n},t}$ to $P_\lambda(q,t) $ and use the identity (6.17) from~\cite{Macdonald1} as above. 
\end{proof}

We are now ready to write a combinatorial formula for the the multiple factorial function $[z, s]_{\lambda}$.
\begin{thm}
\label{combBracket}
Let $z, s\in\mathbb{C}^n$ be $n$-tuples of complex numbers,  and $\lambda$ be a partition of length at most $n$. Then we have
\begin{multline}
[z, s]_{\lambda} = 
\prod_{s\in \lambda} \dfrac{ 1 }{(1- q t^{n-1-l'(s)})  }   
\prod_{s\in \lambda} \left( 
\dfrac{1-q^{a_\lambda(s)} t^{l_\lambda(s)+1}}
{1 - q^{a'_\lambda(s)} t^{- l'_\lambda(s)+n}}  \right)  \\
\sum_{T} \psi_T(q,t) \prod_{s\in\lambda}  t^{ -l'(s) +n -T(s)} 
(1 - s_{T(s)}  q^{z_T(s) -a'(s) } t^{l'(s) }   ) 
\end{multline}
\end{thm}

\begin{proof}
Using the definition of the multiple factorial~(\ref{factorialfunc}), and Lemma~\ref{factors1}, Lemma~\ref{factors2}, and Theorem~\ref{combformw}, we write 
\begin{multline}
[z, s]_{\lambda} 
= \prod_{s\in \lambda} \dfrac{q}{(1- q t^{n-1-l'(s)})  }   
\prod_{s\in \lambda} \left( 
\dfrac{1-q^{a_\lambda(s)} t^{l_\lambda(s)+1}}
{1 - q^{a'_\lambda(s)} t^{- l'_\lambda(s)+n}}  \right)  \\ \cdot
\sum_{T} \psi_T(q,t) \prod_{s\in\lambda} (-q^{-1}) 
(s_{T(s)}  q^{z_T(s) -a'(s) } t^{n-T(s)}  -  t^{ -l'(s) +n -T(s)}   ) 
\end{multline}
Manipulating factors and simplifying yields the desired formula. 
\end{proof}

Next we write a combinatorial formula for the multiple $qt$-binomial coefficients in terms of Young tableau as follows. 
\begin{thm}
\label{combBinomial}
With the notation as above, we have 
\begin{multline}
\binom{z}{\mu}_{\!\!\!q,t} 
= \prod_{s\in \mu} \dfrac{ 1}{ (1-q^{a_\mu(s)+1} t^{l_\mu(s)} ) } \\  \cdot 
\sum_{T} \psi_T(q,t) \prod_{s\in\mu} t^{l'(s)+1 -T(s)}
(1- q^{z_{T(s)} -a'(s) } t^{l'(s) } ) 
\end{multline}
where $z\in\mathbb{C}^n$, and $\mu$ is a partition of at most $n$ parts. 
\end{thm}

\begin{proof}
We substitute $q^z t^{\delta(n)}$ for $z$ in Theorem~\ref{combformw}, and use the definition of the multiple binomial coefficients~(\ref{qtbinomcoeffExt}) together with Lemma~\ref{factors1} and Lemma~\ref{factors2} to write 
\begin{multline}
\binom{z}{\mu}_{\!\!\!q,t} 
= \prod_{s\in \lambda} \dfrac{q t^{2l'(s)+1-n}}{ (1-  q^{1+a'(s)} t^{n-1-l'(s)} ) }   
\prod_{s\in \lambda}  
\left( \dfrac{1 - q^{a'_\lambda(s)+1} t^{-l'_\lambda(s)+n-1}  }
{1-q^{a_\lambda(s)+1} t^{l_\lambda(s)}} \right) \\ \cdot
\sum_{T} \psi_T(q,t) \prod_{s\in\lambda} (-q^{-1}) 
(q^{z_T(s) -a'(s) } t^{n-T(s)}  -  t^{ -l'(s) +n -T(s)}   ) 
\end{multline}
We now cancel common factors and simplify to obtain the formula above. 
\end{proof}

We study combinatorial formulas for two more sequences of combinatorial numbers, namely the multiple Catalan numbers, and the multiple Lah numbers. 

The classical Catalan numbers are defined by the relation
\begin{equation}
C_n:=\dfrac{1}{n+1} \binom{2n}{n}
\label{eq:Catalan}
\end{equation}
A multiple $qt$-analogue of these numbers is defined in terms of the multiple binomial coefficients and the factorial functions as follows (see~\cite{CoskunDiscM}). 
\begin{defn}
Let $\lambda$ be an $n$-part partition. Then the multiple $qt$-Catalan number $C_\lambda$ is defined by 
\begin{equation}
\label{defnCat}
C_\lambda:=\dfrac{1}{[\lambda+1]_{q,t}} \binom{2 \lambda}{\lambda}_{\!\!\!q,t} 
\end{equation}
where $2\lambda=(2\lambda_1, \ldots, 2\lambda_n)$ and $\lambda+1=(\lambda_1+1, \ldots, \lambda_n+1)$.
\end{defn}
Therefore, using the identity~(\ref{binom_rect2}), the definition may be written as
\begin{equation}
C_\lambda:=\prod_{i=1}^n \dfrac{(1-qt^{n-i} )} {(1-q^{\lambda_i+1} t^{n-i} )}
 \binom{2 \lambda}{\lambda}_{\!\!\!q,t} 
\end{equation}
In the special case when $\lambda=\bar k=(k,\ldots,k)$ is a rectangular $n$-part partition, we get a simple product representation 
\begin{equation}
C_{\bar k}:= \prod_{i=1}^n \dfrac{(1-qt^{n-i} )} {(1-q^{k+1} t^{n-i} )}
\prod_{i=1}^n \dfrac{(q^{1+k} t^{n-i})_{k} }  { (qt^{n-i} )_{k} }  
=\prod_{i=1}^n \dfrac{(q^{2+k} t^{n-i})_{k-1} }  { (q^2t^{n-i} )_{k-1} }  
\end{equation}
using the evaluation~(\ref{binom_rect1}) above. In the particular case when $k=1$, 
we get $C_{\bar 1}=1$ in any dimension $n$. 

We like to point out that our multiple $qt$-Catalan numbers appears to be different from the $qt$-Catalan numbers defined in~\cite{GarsiaH1} and developed by~\cite{GarsiaHa1, Haiman1} 
and others. These numbers are one dimensional (i.e., is not a multiple sequence) in the sense that they are indexed by positive integers (the weight of the partitions) as seen in the definition 
\begin{equation}
C_n(q,t) = \sum_{\mu \vdash n} \dfrac{t^{2\sum l} q^{2\sum a}(1-t)(1-q) \prod_{}^{0,0} (1-q^{a'} t^{l'}) \sum q^{a'}t^{l'}}{\prod (q^{a} - t^{l +1})(t^{l} -q^{a+1})}
\end{equation}
where $\prod_{}^{0,0}$ means the product skips the corner cell. 
To avoid any possible confusion, we call our family of numbers above multiple $qt$-Catalan numbers. 
The multiple $qt$-Catalan numbers reduce, for $n=1$, to the one-dimensional $q$-Catalan numbers that are defined in~\cite{Carlitz3} and studied by~\cite{Andrews3, Andrews4, Stembridge1, FurlingerH1} and others. 

We next write a combinatorial formula for the multiple Catalan numbers.

\begin{thm}
Let $\lambda$ be an $n$-part partition, that is $\ell(\lambda)=n$. With the notation as above, we have 
\begin{multline}
C_\lambda 
= \prod_{s\in \bar 1 } \dfrac{ (1- q^{1+a'(s)} t^{n-1-l'(s)} ) }
{ (1- q^{1+\lambda_{1+l'(s)} } q^{a'(s)} t^{n-1-l'(s)} ) } 
\prod_{s\in \lambda} \dfrac{ 1}{ (1-q^{a_\lambda(s)+1} t^{l_\lambda(s)} ) } \\  \cdot 
\sum_{T} \psi_T(q,t) \prod_{s\in\lambda} t^{l'(s)+1 -T(s)}
(1- q^{2\lambda_{T(s)} -a'(s) } t^{l'(s) } ) 
\end{multline}
\end{thm}

\begin{proof}
Using Lemma~\ref{factors2}, and setting $x_i= q^{1-k+x_i} t^{n-1}$ in the identity~(\ref{binom_rect2}), we obtain
\begin{equation}
[\lambda+1]_{qt} = \binom{\lambda+1}{\bar 1}_{\!\!\!q,t} = 
\prod_{s\in \bar 1 } \dfrac{ (1- q^{1+\lambda_{1+l'(s)} } q^{a'(s)} t^{n-1-l'(s)} ) }
{ (1- q^{1+a'(s)} t^{n-1-l'(s)} ) }
\end{equation}
We also set $z=2\lambda$ and $\mu=\lambda$ in Theorem~\ref{combBinomial}, and substitute both results into the definition~(\ref{defnCat}) of $C_\lambda$ to write 
\begin{multline}
C_\lambda=\dfrac{1}{[\lambda+1]_{q,t}} \binom{2 \lambda}{\lambda}_{\!\!\!q,t} =
\prod_{s\in \bar 1 } \dfrac{ (1- q^{1+a'(s)} t^{n-1-l'(s)} ) }
{ (1- q^{1+\lambda_{1+l'(s)} } q^{a'(s)} t^{n-1-l'(s)} ) }
\binom{2 \lambda}{\lambda}_{\!\!\!qt} \\
= \prod_{s\in \bar 1 } \dfrac{ (1- q^{1+a'(s)} t^{n-1-l'(s)} ) }
{ (1- q^{1+\lambda_{1+l'(s)} } q^{a'(s)} t^{n-1-l'(s)} ) } 
\prod_{s\in \lambda} \dfrac{ 1}{ (1-q^{a_\lambda(s)+1} t^{l_\lambda(s)} ) } \\  
\sum_{T} \psi_T(q,t) \prod_{s\in\lambda} t^{l'(s)+1 -T(s)}
(1- q^{2\lambda_{T(s)} -a'(s) } t^{l'(s) } ) 
\end{multline}
This is what we wanted to show. 
\end{proof}

The last sequence of combinatorial numbers we study in this paper is the multiple Lah numbers. The classical Lah numbers are defined to be the connection coefficients in the expansion 
\begin{equation}
x^{\overline{n}} = \sum_{k=0}^n L(n,k) \, x_{\underline{k}}
\end{equation}
where $x_{\underline{n}}=x (x-1) \ldots (x-n+1)$ denotes the falling factorial as before, and $x^{\overline{n}}=x (x+1) \ldots (x+n-1)$ the rising factorial. 
Various $q$-analogues of these numbers are developed in one dimensional case in~\cite{GarsiaR1, Wagner2} and others. 

We defined the multiple $qt$-Lah numbers in~\cite{Coskun3} in terms of the multiple factorial function and its flipped version as follows:
\begin{defn}
Let $[\bar x]^\lambda$ denote the multiple analogue of the rising factorial. That is,
\begin{equation}
\label{rising}
[\bar x]^\lambda = [\bar x, q, t]^\lambda := [\bar x, q^{-1}, t^{-1}]_\lambda
\end{equation}
For an $n$ part partition $\lambda$, the $qt$-Lah numbers $L(\lambda,\mu) = L(\lambda,\mu, q, t) $ are defined by 
\begin{equation}
\label{Lahnumber}
[\bar x]^\lambda  
= \sum_{\mu \subseteq \lambda} (-1)^{|\mu|} q^{-|\mu|+  2n(\mu') }  t^{-n(\mu)}    
L(\lambda,\mu, q, t)  \, [\bar x]_\mu 
\end{equation}
where $x\in\mathbb{C}$, and $\bar x=\{x,\ldots, x\}\in\mathbb{C}^n$ as before. 
\end{defn}

Among many interesting properties obtained in~\cite{Coskun3}, we review a fundamental result that, similar to the  one dimensional case, the multiple $qt$-Lah numbers admit an explicit representation. More specifically, we have 
\begin{multline}
\label{closedLahexplicit}
L(\lambda, \mu)  = (-1)^{|\lambda|+|\mu|} q^{-|\lambda|+|\mu|} t^{n(\lambda) - n(\mu)} \\
\cdot \prod_{i=1}^n  \left\{ (1-qt^{n-i} )^{-\lambda_i+\mu_i}  \right\}  
\dfrac{ (t^{2(n-1)})_{\lambda}} {(t^{2(n-1)} )_{\mu} }  \,
\binom{\lambda}{\mu}_{\!\!\!q,t}  \hspace{30pt}
\end{multline}

We conclude this article by providing a combinatorial formula for the multiple Lah numbers in the next result. 
\begin{thm}
Let $\lambda$ and $\mu$ be partitions of length at most $n$. With the notation as above, we have
\begin{multline}
L(\lambda ,\mu) 
= \prod_{s\in \lambda/\mu} \dfrac{t^{l'(s)} (1- q^{a'(s)}t^{2(n-1)-l'(s)} ) }
{ (-q) (1- q t^{n-1-l'(s)}) } 
\prod_{s\in \mu} \dfrac{ 1}{ (1-q^{a_\mu(s)+1} t^{l_\mu(s)} ) }  \\ \cdot 
 \sum_{T} \psi_T(q,t) \prod_{s\in\mu} t^{l'(s)+1 -T(s)}
(1- q^{\lambda_{T(s)} -a'(s) } t^{l'(s) } ) 
\end{multline}
\end{thm}

\begin{proof}
We use the explicit formula for the multiple Lah numbers~(\ref{closedLahexplicit}) together with  Lemma~\ref{factors1}, Lemma~\ref{factors2}, and Theorem~\ref{combBinomial} to write 
\begin{multline}
L(\lambda ,\mu) 
=(-1)^{|\lambda|+|\mu|} q^{-|\lambda|+|\mu|} t^{n(\lambda) - n(\mu)} 
\prod_{i=1}^n  (1-qt^{n-i} )^{-\lambda_i+\mu_i}   
\dfrac{ (t^{2(n-1)})_{\lambda}} {(t^{2(n-1)} )_{\mu} }  \,
\binom{\lambda}{\mu}_{\!\!\!q,t}  \\
= 
\dfrac{\prod_{s\in \lambda} (1- q^{a'_\lambda(s)}t^{2(n-1)-l'_\lambda(s)} ) 
\prod_{s\in \mu} (1- q t^{n-1-l'_\mu(s)})  }
{\prod_{s\in \mu} (1-  q^{a'_\mu(s)}t^{2(n-1)-l'_\mu(s)} ) 
\prod_{s\in \lambda} (1- q t^{n-1-l'_\lambda(s)})  }  
\prod_{s\in \mu} \dfrac{ 1}{ (1-q^{a_\mu(s)+1} t^{l_\mu(s)} ) } \\
\cdot \sum_{T} \psi_T(q,t) \prod_{s\in\mu} t^{l'(s)+1 -T(s)}
(1- q^{\lambda_{T(s)} -a'(s) } t^{l'(s) } ) 
\end{multline}
which simplifies to the result above.
\end{proof}

\section{Conclusion }
We derived explicit combinatorial formulas for the multiple $qt$-factorial function, multiple $qt$-binomial coefficients, multiple $qt$-Catalan numbers, and multiple $qt$-Lah numbers in terms of semistandard reversed Young tableau in the present paper. 
We will construct combinatorial formulas for other number sequences in a future publication, and investigate other possible combinatorial interpretations.

\end{document}